\documentclass[11pt]{article}
\usepackage{amsfonts,amsmath,bm, xspace}
\usepackage{multirow,graphicx, algorithm,algorithmic,rotating}
\usepackage{url}
\usepackage{fullpage}
\usepackage[small,bf]{caption}
\newtheorem{theorem}{Theorem}[section]
\newtheorem{lemma}[theorem]{Lemma}
\newtheorem{corollary}[theorem]{Corollary}
\newtheorem{proposition}[theorem]{Proposition}

\def \endprf{\hfill {\vrule height6pt width6pt depth0pt}\medskip}
\newenvironment{proof}{\noindent {\bf Proof} }{\endprf\par}

\numberwithin{equation}{section}

\newcommand{\C}{\mathbb{C}}
\newcommand{\D}{\mathbb{D}}

\newcommand{\bbS}{\mathbb{S}}

\newcommand{\cA}{\mathcal{A}}

\newcommand{\cH}{\mathcal{H}}
\newcommand{\cL}{\mathcal{L}}

\newcommand{\R}{\mathbb{R}}
\newcommand{\E}{\operatorname{\mathbb{E}}}

\newcommand{\trace}{\operatorname{trace}}

\newcommand{\minimize}{\mbox{minimize}}
\newcommand{\st}{\mbox{subject to}}

\newcommand{\eq}[1]{(\ref{eq:#1})}

\title{Linear System Identification via Atomic Norm Regularization}

\author{Parikshit Shah, Badri Narayan Bhaskar, Gongguo Tang and Benjamin Recht\\
University of Wisconsin-Madison}

\date{}

\begin{document}

\maketitle

\vspace{-0.3in}

\bibliographystyle{plain}
\begin{abstract}
This paper proposes a new algorithm for linear system identification from noisy measurements.  The proposed algorithm balances a data fidelity term with a norm induced by the set of single pole filters. We pose a convex optimization problem that approximately solves the atomic norm minimization problem and identifies the unknown system from noisy linear measurements.  This problem can be solved efficiently with standard, freely available software.  We provide rigorous statistical guarantees that explicitly bound the estimation error (in the $\cH_2$-norm) in terms of the stability radius, the Hankel singular values of the true system and the number of measurements. These results in turn yield complexity bounds and asymptotic consistency. We provide numerical experiments demonstrating the efficacy of our method for estimating linear systems from a variety of linear measurements.
\end{abstract}

{\bf Keywords} System identification. Atomic norms.  Hankel operators. Optimization.

\section{Introduction}\label{sec:intro} \noindent

Identifying dynamical systems from noisy observation of their input-output behavior is of fundamental importance in systems and control theory. Often times models derived from physical first principles are not available to the control engineering, and computing a surrogate model from data is essential to the design of a control system.  System identification from data is thus ubiquitous in problem domains ranging from process engineering, dynamic modeling of mechanical and aerospace systems, and systems biology. Though there are a myriad of approaches and excellent texts on the subject (see, for example~\cite{LjungBook}), there is still no universally agreed upon approach for this problem.  One reason is that quantifying the interplay between system parameters, measurement noise, and model mismatch tends to be challenging.

This paper draws novel connections between contemporary high-dimensional statistics, operator theory, and linear systems theory to prove consistent estimators of linear systems from small measurement sets.  In particular, building on recent studies of \emph{atomic norms} in estimation theory~\cite{CRPW10,BhaskarAllerton11}, we propose a penalty function which encourages estimated models to have small McMillan degree.

A related family of system identification techniques use finite sample Hankel matrices to estimate dynamical system models, using either singular value decompositions (e.g,~\cite{Verhaegen92,Overschee94}) or semidefinite programming~\cite{Fazel01,Liu08,Smith12,Fazel11}. In all of these techniques, no statistical guarantees were given about the quality of estimation with finite noisy data, and it was difficult to determine how sensitive these methods were to the hidden system parameters or measurement noise.  Moreover, since these problems were dealing with finite, truncated Hankel matrices, it is never certain if the size of the Hankel matrix is sufficient to reveal the true McMillan degree.  Moreover, the techniques based on semidefinite programming are challenging to scale to very large problems, as their complexity grows superlinearly with the number of measurements.

In contrast, the atomic norm regularizer proposed in this paper is not only equivalent to the sum of the Hankel singular values (the Hankel nuclear norm), but is also well approximated by a finite dimensional, $\ell_1$ minimization problem. We  show that solving least-squares problems regularized by our atomic norm is consistent, and scales gracefully with the stability radius, the McMillan degree of the system to be identified, and the number of measurements.  Our numerical experiments validate these theoretical underpinnings, and show that our method has great promise to provide concrete estimates on the hard limits of estimating linear systems.


\subsection{Notation}\label{sec:notation}
We adopt standard notation; $\D$ and $\bbS$ will denote respectively the open unit ball and the unit
circle in the complex plane $\C$ . $\cH_2$ and $\cH_\infty$ will denote the Hardy spaces of functions analytic outside $\D$, with the norms
\[\|f\|_{\cH_2} = \tfrac{1}{2\pi} \int_0^{2\pi} |f(e^{i\theta})|^2 d\theta~~~~~\mbox{and}~~~~~\|f\|_{\cH_\infty} = \sup_{z\in\bbS} |f(z)|\]
respectively.  $\ell_2([a,b])$ will denote the set of square summable sequences on the integers in $[a, b]$.

\section{Atomic Decompositions of Transfer Functions}\label{sec:atomic-def}
We restrict our attention to SISO systems in this manuscript, as this will simplify the presentation.  However, we will describe in the discussion how to extend our techniques to MIMO systems.  Suppose we wish to estimate a SISO, LTI system with transfer function $G_\star(z)$ from a finite collection of measurements $y=\Phi(G_\star)$.  The set of all transfer functions is an infinite dimensional space, so reconstructing $G_\star$ from this data is ill-posed.  In order to make it well posed, a common regularization approach constructs a penalty function $\operatorname{pen}(\cdot)$ that encourages ``low-complexity'' models and solves the optimization problem
\begin{equation}\label{eq:regularized}
	\minimize_G~\|\Phi(G)-y\|_2^2 + \mu \operatorname{pen}(G) \,.
\end{equation}
This formulation uses the parameter $\mu$ to balance between model complexity and fidelity to the data.  The least-squares cost can be modified to other convex loss functions if knowledge about measurement noise is available (as in~\cite{Smith12,Paganini96}), though in general it is less clear how to design a good penalty function.

In many applications, we know that the true model can be decomposed as a linear combination of very simple building blocks.  For instance, sparse vectors can be written as short linear combinations of vectors from some discrete dictionary and low-rank matrices can be written as a sum of a few rank-one factors.  In~\cite{CRPW10}, Chandraskearan et al. proposed a universal heuristic for constructing regularizers based on such prior information.  If we assumed that 
\[
	G_\star = \sum_{i=1}^r c_i a_i\,,~\mbox{for some}~a_i\in\cA,c_i\in \C\,,
\]
where $\cA$ is an origin-symmetric set of ``atoms'' normalized to have unit norm and $r$ is relatively small, then the appropriate penalty function is the guage function (or the Minkowski functional) induced by the atomic set $\cA$:
\begin{equation}\label{eq:atomic-norm-def}
\begin{split}
	\|G\|_{\cA}: & = \inf\left\{ t \; : \; G \in t\text{ conv}(\cA)  \right\}  =\inf\left\{ \sum_{a\in \cA} |c_a|~:~G = \sum_{a\in \cA} c_a a\right\}\,.
\end{split}
\end{equation}
In~\cite{CRPW10}, it is shown that minimizing the atomic norm subject to compressed measurements yielded the tightest known bounds for recovering many classes of models from linear measurements. Moreover, in~\cite{BhaskarAllerton11}, the atomic norm regularizer was studied in the context of denoising problems and was found to produce consistent estimates at nearly optimal estimation error rates for many classes of atoms.

To apply these atomic norm techniques to system identification, we must first determine the appropriate set of atoms.  For discrete time LTI systems with small McMillan degree, we can always decompose any finite dimensional, strictly proper system $G(z)$ as:
$$
G(z)=\sum_{i=1}^{s} \frac{c_i}{z-a_i}\,.
$$
via a partial fraction expansion.  Hence, it makes sense that our set of atoms should be single-pole transfer functions.  We propose the following atomic set for linear systems
\[
\cA = \left\{ \varphi_{w}(z) = \frac{1-|w|^2}{z-w} ~:~w\in \D\right\}\,.
\]
The numerator is normalized so that the Hankel norm of each atom is $1$.  See the discussion in Section~\ref{sec:hankel} for precisely why this normalization is desirable.

The atomic norm penalty function associated with these atoms is 
\begin{equation}\label{eq:atomic-def}
	\| G(z) \|_{\cA} = \inf \left\{ \sum_{w\in \D} |c_w| ~:~ G(z) = \sum_{w\in\D} \frac{c_w (1-|w|^2)}{z-w}\right\}\,,
\end{equation}
where the summation implies that only a countable number of terms have nonzero coefficients $c_w$.  This expression finds the decomposition of $G(z)$ into a linear combination of single pole systems such that the $\ell_1$ norm, weighted by the norms of the single poles, is as small as possible.

With this penalty function in hand, we now turn to analyzing its utility.  In Section~\ref{sec:hankel}, we first show that for most systems of interest $\|G\|_{\cA}$ is a well-defined, bounded quantity.  Moreover, we will show that the atomic norm is equivalent to the nuclear norm of the Hankel operator associated with  $G$.  Hence, the models that are preferred by our penalty function will have low-rank Hankel operators, and thus low McMillan degrees.

In Section~\ref{sec:computation}, we turn to computation, demonstrating practical algorithms for approximating atomic norm regularization problems for several classes of measurements.  We will show that with finite data, our atomic norm minimization problem is well-approximated by a finite-dimensional $\ell_1$ norm regularization problem.  In particular, using specialized algorithms adapted to the solution of the LASSO~\cite{Wright09}, we can solve atomic norm regularization problems in time competitive with respect to techniques that regularize with the nuclear norm and SVD-based subspace identification methods.

Finally, we analyze the statistical performance of atomic norm minimization in Section~\ref{sec:statistics}.  We show that our algorithm is asymptotically consistent over several measurement ensembles of interest.  We focus on sampling the transfer function on the unit circle and present $\cH_2$ error bounds in terms of the stability radius, Hankel singular values, $\cH_{\infty}$ norm, and McMillan degree of the system to be estimated.

\section{The Hankel Nuclear Norm and Atomic Norm Minimization}\label{sec:hankel}

Let us first show that most LTI systems of interest do indeed have finite atomic norm, and, moreover, that the atomic norm is closely connected with the sum of the Hankel singular values.

\subsection{Preliminaries: the Hankel operator}\label{sec:hankel-defs}
Recall that the \emph{Hankel operator}, $\Gamma_G$, of the transfer function $G$ is defined as the mapping from the past to the future under the transfer function $G$.  Given a signal $u$ supported on $(-\infty,-1]$, the output under $G$ is given by $g * u$ where ``$*$'' denotes convolution and $g$ is the impulse response of $G$: 
\[
	G(z) = \sum_{k=1}^\infty g_k z^{-k}\,.
\]
$\Gamma_G$ is then simply the projection of $g*u$ onto $[0,\infty)$.  An introduction to Hankel operators in control theory can be found in~\cite[Chapter 4]{DullerudPaganiniBook} or~\cite[Chapter 7]{Zhou95}.

The \emph{Hankel norm} of $G$ is the operator norm of $\Gamma_G$ considered as an operator mapping $\ell_2(-\infty,-1]$ to $\ell_2[0,\infty)$. The \emph{Hankel nuclear norm} of $G$ is the nuclear norm (aka the trace norm or Schatten $1$-norm) of $\Gamma_G$.  To be precise, an operator $T$ is in the \emph{trace class} $S_1$ if the trace of $(T^*T)^{1/2}$ is finite.  This implies first that $T$ is a compact operator and admits a singular value decomposition
\begin{equation*}
    T(f) = \sum_{i=1}^\infty \sigma_i \langle v_i, f \rangle u_i \,.
\end{equation*}
The sequence $\sigma_i$ are called the \emph{Hankel singular values} of $T$.  Moreover, the Schatten $1$-norm of $T$ is given by
\begin{equation*}
    \|T\|_1 = \trace\left((T^*T)^{1/2}\right) =
    \sum_{i=1}^\infty \sigma_i\,.
\end{equation*}

\subsection{The atomic norm is equivalent to the Hankel nuclear norm}
The rank of the Hankel operator determines the McMillan degree
of the linear system defined by $G$.  Rank minimization is notoriously computationally challenging (see~\cite{Recht10} for a discussion), and we don't expect to be able to directly penalize the norm of the Hankel operator in implementations.  Thus, as is common, a reasonable heuristic for minimizing the rank of the Hankel operator would be to minimize the sum of the Hankel singular values, i.e., to minimize the Schatten $1$-norm of the Hankel operator.  For rational transfer functions, we can compute the Hankel nuclear norm via a balanced realization~\cite{Zhou95}.  On the other hand, while the maximal Hankel singular value can be written variationally as an LMI, we are not aware of any such semidefinite programming formulations for the Hankel nuclear norm.  

The following theorem provides a path towards minimizing the Hankel nuclear norm, minimizing the atomic norm $\|G(z)\|_{\cA}$ as a proxy.  Indeed, from the view of Banach space theory, the atomic norm is~\emph{equivalent} to the Hankel nuclear norm.

\begin{theorem}\label{thm:hankel-nuclearity}
Let $G \in \cH_2$.  Then $\Gamma_G$ is trace class if and only if there exists a sequence $\{\lambda_k\} \in \ell_1$ and a
    sequence $\{w_k\}$ with $w_k \in \D$ such that
\begin{equation}\label{eq:kernel-form}
    g(z) = \sum_{i=1}^\infty \lambda_k \frac{1-|w_k|^2}{z-w_k}\,.
    \end{equation}
Moreover, we have the following chain of inequalities
\begin{equation}\label{eq:norm-ineq-chain}
\tfrac{\pi}{8}\|G\|_{\cA} \leq  \|\Gamma_G\|_1 \leq
\|G\|_{\cA} 
\end{equation}
where  $\|G\|_{\cA}$ is given by~\eq{atomic-norm-def}.
\end{theorem}
\emph{Proof Outline}  Theorem~\ref{thm:hankel-nuclearity} follows by carefully combining several different results from operator theory. Peller first showed that transfer functions with trace class Hankel operators formed a \emph{Besov space}~\cite{Peller79}. Peller's argument can be found in his book~\cite{PellerHankelBook}. The atomic decomposition of such operators is due to Coifman and Rochberg~\cite{Coifman80}. The norm bounds~\eq{norm-ineq-chain} were proven by Bonsall and Walsh~\cite{Bonsall86}. There they show that the $\tfrac{\pi}{8}$ is the best possible lower bound.  They also show that if $\|\Gamma_g\|_1\leq C\|g\|_{\cA}$ for all $g$, then $C$ must be at least $\tfrac{1}{2}$, so the chain of inequalities is nearly optimal. A concise presentation of the full argument can be found  in~\cite{PartingtonHankelBook}. A modern perspective using the theory of reproducing kernels can be found in~\cite{ZhuBook}. 
\vspace{1mm}

\noindent Theorem~\ref{thm:hankel-nuclearity} asserts that a transfer function has a finite atomic norm if and only if the sum of its Hankel singular values is finite.  In particular, this means that every rational transfer function has a finite atomic norm.  More importantly, the atomic norm is equivalent to the Hankel nuclear norm.  Thus if we can approximately solve atomic norm-minimization, we can approximately solve Hankel nuclear norm minimization and vice-versa.  We now turn to such computational considerations.

\section{Algorithms for atomic norm minimization}\label{sec:computation}

From here on, let us assume that the $G_\star$ that we seek to estimate has all of its poles of magnitude at most $\rho$  ( we will call $\rho$ the \emph{stability radius}, and treat it as a known parameter).  Let $\D_\rho$ denote the set of all complex numbers with norm at most $\rho$.  Note that if $G_\star$ has stability radius $\rho$ then
\[
	\|G\|_{\cA}:=\inf\left\{ \sum_{w\in \D_\rho} |c_w|~:~G(z) = \sum_{w\in \D_\rho} \frac{c_w (1-|w|^2)}{z-w}\right\}\,.
\]
That is, we can restrict our set of atoms to only be those single pole systems with stability radius equal to $\rho$.  For the remainder of this manuscript, we assume that $\cA$ only consists of such single pole systems.

In what follows, we focus our attention on linear measurement maps.  Let $\cL_i: \cH \mapsto \C $ be a linear functional that serves as a measurement operator for the system $G(z)$.  Many maps of interest can be phrased as linear functionals of the transfer function,
\begin{enumerate}
	\item Samples of the frequency response $\mathcal{L}_k(G):=G(e^{i\theta_k})$ for $k=1,\ldots, n$.  From a control theoretic perspective, this measurement operator corresponds to measuring the gain and phase of the linear system at different frequencies.
	\item Samples of the impulse response, $\mathcal{L}_k(G):=g_{i_k}$ for $k=1,\ldots, n$ and $i_k \in [1,\infty)$.
	\item Convolutions of the impulse response with a pseudorandom signal $u_k$: $\mathcal{L}_k(G):=\sum_{j=1}^{\infty} g_j u_{k-j}$.
\end{enumerate}

In all of these cases, we consider the problem
\begin{equation}\label{eq:atomic-norm-lin-inverse}
	\minimize_{G} \tfrac{1}{2} \sum_{i=1}^n |\cL_i(G) - y_i |^2 + \mu \|G\|_{\cA}\,.
\end{equation}
This problem is equivalent to the constrained, semi-infinite programming problem
\[
	\begin{array}{ll}
	\minimize_{x,G} & \tfrac{1}{2}\sum_{k=1}^n |x_k - y_k |^2 + \mu \sum_{w\in \D_\rho} |c_w|\\
	\st & x_k = \cL_k(G)~~\mbox{for}~i=1,\ldots, n\\
	& G = \sum_{w\in \D_\rho} \frac{c_w(1-|w|^2)}{z-w}
	\end{array}
\]
Eliminating the equality constraint gives yet another equivalent formulation
\begin{equation}\label{eq:atomic-norm-lin-inverse-finite-dim}
	\begin{array}{ll}
	\minimize_{x} & \tfrac{1}{2}\sum_{k=1}^n |x_k - y_k |^2 + \mu \sum_{w\in \D_\rho} |c_w|\\
	\st & x_k = \sum_{w\in \D_\rho} c_w \cL_k\left(\frac{1-|w|^2}{z-w}\right)~~\mbox{for}~i=1,\ldots, n\,.
	\end{array}
\end{equation}
Note that in this final formulation, our decision variable is $x$, a finite dimensional vector, and $c_w$, the coefficients of the atomic decomposition.  The infinite dimensional variable $G$ has been eliminated.  Let us define a norm on $\R^n$ based on the formulation~\eq{atomic-norm-lin-inverse-finite-dim}
\[
	\|x\|_{\cL(\cA)} = \inf \left\{ \sum_{w\in \D_\rho} |c_w| ~:~ x_i = \sum_{w\in\D_\rho} c_w\mathcal{L}_i\left(\frac{1-|w|^2}{z-w} \right)\right\}\,.
\]
Then we see that problem~\eq{atomic-norm-lin-inverse} is equivalent to the denoising problem
\begin{equation}\label{eq:finite-dim-ast}
	\minimize_{x} \tfrac{1}{2}\|x-y\|_2^2 + \mu \|x\|_{\cL(\cA)}\,.
\end{equation}
Note that the first term is simply the squared Euclidean distance between $y$ and $x$ in $\R^n$.  The second term is an atomic norm on $\R^n$ induced by the linear map of the set of transfer functions via the measurement operator $\cL$.  In order to tractably solve~\eq{atomic-norm-lin-inverse}, we thus only need focus on computational schemes for computing or approximating $\|x\|_{\cL(\cA)}$.  The following proposition asserts that we can approximate this finite dimensional atomic norm via a sufficiently fine discretization of the unit disk.

\begin{proposition}\label{prop:grid}
Let $\D_\rho^{(\epsilon)}$ be a finite subset of the unit disc such that for any $w \in \D_\rho$ there exists a $v \in \D_\rho^{(\epsilon)}$ satisfying $|w-v| \leq \epsilon$.  Define
\[
	\|x\|_{\cL(\cA_\epsilon)} = \inf\left\{ \sum_{w\in \D_{\rho}^{(\epsilon)}} |c_w| ~:~ x_i = \sum_{w\in\D_{\rho}^{(\epsilon)}} c_w\mathcal{L}_i\left(\frac{1-|w|^2}{z-w} \right)\right\}\,.
\]
Then there exists a constant $C_\epsilon \in [0,1]$ such that
\[
	C_\epsilon \|x\|_{\cL(\cA_{\epsilon})} \leq \|x\|_{\cL(\cA)} \leq \|x\|_{\cL(\cA_{\epsilon})} \,.
\]
\end{proposition}
\noindent The set $\D_\rho^{(\epsilon)}$ is called an $\epsilon$-net for the set $\D_\rho$.  We show in the appendix that when $\mathcal{L}_k(G) = G(e^{i\theta_k})$, $C_{\epsilon}$ is at least $(1-\tfrac{16 \rho \epsilon}{\pi(1-\rho)})$.  Other measurement ensembles can be treated similarly.

When we replace $\|x\|_{\cL(\cA)}$ with its discretized counterpart $\|x\|_{\cL(\cA)}$ in~\eq{finite-dim-ast}, 
\[
	\minimize_{x} \tfrac{1}{2}\|x-y\|_2^2 + \mu \|x\|_{\cL(\cA_\epsilon)}
\]
is equivalent to
\begin{equation}\label{eq:DAST}
	\minimize_{c}  \tfrac{1}{2}\|Mc-y\|_2^2 + \mu \sum_{w\in \D_\rho^{(\epsilon)}} |c_w|
\end{equation}
where
\[
	M_{ij} = \mathcal{L}_i \left(\tfrac{1-|w_j|^2}{z-w_j}\right)
\]
and $j$ indexes the set $\D_{\rho}^{(\epsilon)}$. That is $M$ is an $n \times |\D_\rho^{(\epsilon)}|$ matrix.  Problem~\eq{DAST} is a weighted $\ell_1$ regularization problem with real or complex data depending on specific problem.  We call~\eq{DAST} \emph{Discretized Atomic Soft Thresholding} (DAST), as coined in~\cite{BhaskarAllerton11}.

The DAST problem can be solved very efficiently with a variety of off-the-shelf tools including SPARSA~\cite{Wright09}, FPC~\cite{Hale08} or even more general purpose packages such as YALMIP~\cite{YALMIP} or CVX~\cite{cvx}.  DAST yields an approximate solution to problem~\eq{atomic-norm-lin-inverse}, and, as we will see, yields a statistically consistent estimate provided the parameter $\epsilon$ is adjusted to meet the desired numerical accuracy.

\section{Statistical Bounds}\label{sec:statistics}
Let $\cL_i: \cH \mapsto \C $ be a linear functional that serves as a measurement operator for the system $H(z)$.  In this section, let us suppose that we obtain noisy measurements of the form
$$
y_i=\cL_i \left( H(z) \right) +\omega_i \qquad i=1, \ldots, n.
$$
where $\omega_i$ is a noise sequence consisting of independent, identically distributed random variables.  In this section, we will specialize our results to the case where $\cL$ returns samples from the frequency response at uniformly spaced frequencies:
$$
\cL_k(H(z))=H(z_k), \qquad z_k=e^{\frac{2 \pi i k}{m} }, \; k=1, \ldots, n.
$$
While the techniques here extend to other measurement ensembles, they will be explored in a longer version of the paper.
%
%

Our goal in this section is to prove that solving the DAST optimization problem yields a good approximation to the transfer function we are probing. The following theorem provides a precise statistical guarantee on the performance of our algorithm.
\begin{theorem}\label{thm:estimation}  Let $G_\star$ be a strictly proper transfer function with bounded Hankel nuclear norm. Suppose the noise sequence $\omega_i$ is i.i.d. Gaussian with mean zero and variance $\sigma^2$.  Choose $\delta \in (0,1)$ and set $\epsilon = \frac{\pi (1-\rho)\delta}{16 \rho}$.  Let $\D_\rho^{(\epsilon)}$ be as in Proposition~\ref{prop:grid} and let $\hat{c}$ be the optimal solution of~\eq{DAST} with 
\[
\mu=2\sigma \sqrt{ n \log \left(\frac{11 \rho^2}{\delta(1-\rho)} \right)}\,.
\]
Set $\hat{G}(z) = \sum_{w \in \D_{\rho}^{(\epsilon)}} \hat{c}_w \frac{1-|w|^2}{z-w}$.  Then if the set of vectors $\{\cL(\varphi_a)\in \R^n~:~a\in\D_{\rho}^{(\epsilon)}\}$ spans $\R^n$, we have 
\[
\begin{aligned}
	&\|\hat{G}(z)-G_{\star}(z)\|_{\cH_2}^2 \leq  
	186 \frac{1+\rho}{1-\rho} \left( \sqrt{\sigma^2 \log\left(\frac{11 \rho^2}{(1-\rho)\epsilon}\right)}\sqrt{\frac{\|\Gamma_{G_\star}\|_{1}^2}{n(1-\delta)^2} } + \frac{4\|\Gamma_{G_\star}\|_{1}^2}{\pi n(1-\delta)^2}   \right)
	\end{aligned}
\]
with probability $1-e^{-o(n)}$.
\end{theorem}
\begin{corollary} \label{cor:1}
There is a quantity $C$ depending on $\rho$ and $\sigma$ such that for sufficiently large $n$
$$\|\hat{G}(z)-G_{\star}(z)\|_{\cH_2}^2\ \leq C \| \Gamma_{G_{\star}} \|_1 n^{-\frac{1}{2}}$$
with probability exceeding $1-e^{-o(n)}$.
\end{corollary}
\noindent Before we describe how to prove this theorem and its corollary, let us first unpack the features.  First of all, the right hand side is a parameter of the number of samples, the Hankel nuclear norm of the true system, and the stability radius of the true system.  Also, if the McMillan degree of $G_\star(z)$ is $d$, then we can upper bound the Hankel nuclear norm by the product of the McMillan degree and the Hankel norm of $G_\star$: $\|\Gamma_{G_\star}\|_1 \leq d \|\Gamma_{G_\star}\|$.
Second, note that as $n$ tends to infinity, the right hand side tends to zero.  In particular, this means that our discretized algorithm is consistent, and we can quantify the worst case convergence rate.  

The proof of theorem~\ref{thm:estimation} is provided in the appendix.  We prove this theorem by first upper bounding the $\cH_2$ in terms of the mean square error on the observed samples.  We show that this empirical mean square error can be upper bounded in terms of the Hankel nuclear norm of $G_\star$ times the \emph{dual atomic norm} of the noise sequence $\omega$.  We estimate this norm, and in the process compute the optimal value of the regularization parameter.  Putting these pieces together yields our main result.

\section{Numerical Experiments}\label{sec:experiments}
In this section we validate the proposed framework via some preliminary numerical experiments conducted in MATLAB. In many of the experiments where the solution of convex optimization problems was required, the software package CVX \cite{cvx} was used. Throughout our experiments, the discretization of the unit circle was held to approximately $2000$ points.

In the first experiment we consider a stable system $G$ with two poles. We make $m=80$ noisy observations of the frequency response by evaluating the transfer functions $G(z_j)$ at regularly spaced frequencies $z_j=e^{i\theta_j}$ on the complex unit circle. The noise is additive i.i.d. zero-mean Gaussian with a variance of $\sigma^2=10^{-4}$. We reconstruct $\hat{G}(z)$ by DAST as proposed in section \ref{sec:computation}. Our algorithm recovers a system of degree $6$ which achieves an $\mathcal{H}_2$ performance error of $.0043$ and $\mathcal{H}_{\infty}$ error of $.0079$. The locations of the true and recovered poles are depicted graphically in Fig. \ref{fig:1}.

\begin{figure}
  \begin{center}
    \includegraphics[scale=0.4,trim=0mm 60mm 0mm 50mm, clip]{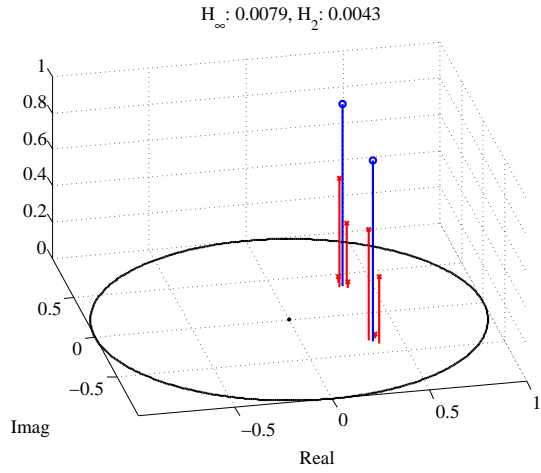} 
  \end{center}
  \caption{In the figure above, the locations marked with a circle represent the locations of the poles (in the complex plane) of a second order discrete time LTI system. The locations marked with a cross correspond to poles recovered by DAST.}
\label{fig:1}
\end{figure}

In Fig. \ref{fig:2} we consider again the problem of recovering a second order system from noisy frequency response measurements. The noise variance is set to $\sigma^2=10^{-4}$. The plot below shows the performance of DAST as the number of measurements increases. The error metric used is the $\cH_2$ norm.
\begin{figure}
  \begin{center}
    \includegraphics[scale=0.4,trim=0mm 0mm 0mm 10mm, clip]{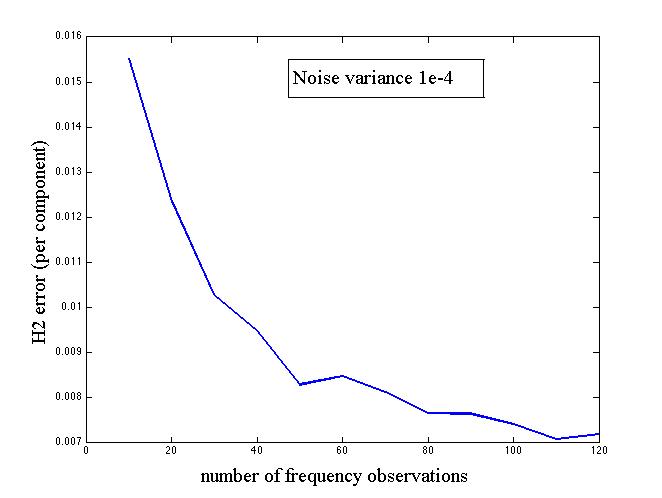} 
  \end{center}
  \caption{In the figure above, we plot the $\cH_2$ estimation error between the true system and the recovered system as the number of frequency measurements $m$ varies.}
\label{fig:2}
\end{figure}

In Fig. \ref{fig:3}, we compare our algorithm to a widely used method known as \emph{subspace identification}  \cite[Chapter 10]{LjungBook}. A second order system, starting from an initial condition of $x[0]=0$ is excited by a random input $u[t]$ corresponding to an i.i.d. sequence of zero-mean, unit-variance Gaussian random variables for $m$ time steps. We record the output $y[t]$ of the system for $m$ time steps. From this input-output relationship, we use DAST and subspace identification to attempt to reconstruct the unknown system. We plot  the estimation error in the $\cH_2$ norm as $m$ is increased from $10$ time units to $120$ time units. As is evident, the performance of DAST is superior to that of subspace identification when $m$ is small, i.e. of the order of $10$ to $50$ measurements.

Another aspect that we emphasize is that in these experiments, subspace identification was assisted with the knowledge of the true system order. If the wrong model order was used, the performance of subspace identification worsened noticeably. By contrast, DAST does not need knowledge of the true system order.

\begin{figure}
  \begin{center}
    \includegraphics[scale=0.4,trim=0mm 0mm 0mm 10mm, clip]{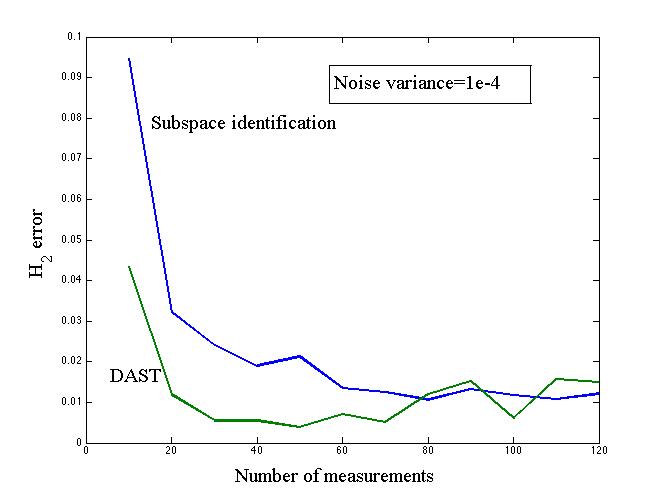} 
  \end{center}
  \caption{In the figure above, we compare $\cH_2$ estimation error of the algorithm proposed in this paper (DAST) and the error obtained by the subspace identification method.}
\label{fig:3}
\end{figure}

\section{Conclusion}\label{sec:conclusion}

By using the atomic norm framework of~\cite{CRPW10}, we were able to posit a reasonable regularizer for linear systems, understand the computational demands of such a regularizer, and analyze its statistical performance.  Since it is closely connected to the Hankel nuclear norm but is computationally more practical, we believe that our atomic norm will be useful in a variety of practical implementations and also in theoretical analysis.  However, there are still several outstanding questions to address before we fully understand the potential of this norm.  We list several of these open problems here.

\paragraph{Other measurement ensembles} Our analysis in Section~\ref{sec:statistics} focused on the particular case of sampling the frequency response at regular intervals.  By focusing on this example, we were able illustrate the critical ingredients to computing convergence rates.  First, we needed to show that our measurement error provided a reasonable upper bound on the distance to the true transfer function.  Second, we used convex analysis to upper bound the measurement error in terms of the statistics of the noise process.  Third, we estimated the noise statistics using probabilistic techniques and appealing to the structure of the atomic set and its $\epsilon$-nets.  This methodology can be extended to the other sampling methods described in Section~\ref{sec:computation}, and may also be extendable to estimating transfer functions from pairs of input-output time series.

\paragraph{Fast Rates and minimax optimality}  The rates provided by Theorem~\ref{thm:estimation} demonstrate that the DAST algorithm is asymptotically consistent.  However, we believe the upper bound we have derived is quite crude.  In particular, as discussed in~\cite{BhaskarAllerton11}, it may very well be possible to improve our upper bounds by leveraging more of the geometry of the set of single-pole transfer functions.  It would be interesting to find reasonable lower-bounds on the reconstruction error from limited measurements, and to see how close we can match these worst-case estimates via a new analysis.

\paragraph{Interpolating with derivatives}  While gridding the space of poles enables us to quickly solve atomic norm problems, a main drawback is that we are then can never exactly localize the true poles of the system without an extremely fine grid.  One recent proposal to enable such a localization uses a linearization technique to simultaneously fit a model on the grid points and at the derivatives of the transfer functions at these grid points~\cite{Simoncelli11}.  It would be of interest to see if such an adjoining method could work in this setting, and future experiments will evaluate the improvements on system identification in theory and in practice.

\paragraph{Extension to MIMO Systems}  While we focused on the single-input single-output (SISO) case in this paper, we expect that these techniques would extend to the multi-input multi-output (MIMO) case. One simple extension to note is that if the number of inputs and outputs in the system remain small, the SISO techniques presented herein could be applied to each input-output pair. Alternative approaches that avoid this pairwise identification would be important in large-scale systems with many inputs and outputs, this is an important direction for future research.

To extend our methodology to MIMO systems, we need to find an appropriate set of atoms. One possibility is the set
\[
\cA = \left\{ \frac{(1-|w|^2) cb^* }{z-w} ~:~w\in \D, c\in\mathbb{S}^{r-1},\,\,b\in\mathbb{S}^{p-1} \right\}\,,
\]
where $p$ is the number of inputs and $r$ is the number of outputs.  Any discrete time, MIMO system with  McMillan degree $d$ can be written in terms of $d$ of these atoms.  The only difficulty remains computing or approximating the norm $\|u\|_{\cL(\cA)}$ for finite measurements as described in Section~\ref{sec:computation}.

\bibliography{hankel_cdc}

\appendix

\section{Useful Lemmas} Before we proceed with the proofs of our main results, let us record a few useful lemmas.  Recall that our atomic functions are defined as
\[
	\varphi_a(z) = \frac{1-|a|^2}{z-a}
\]

\begin{lemma}\label{lemma:atom-hinf}
For any $a \in D$, 	$\|\varphi_a(z)\|_{\cH_\infty} \leq 2$.
\end{lemma}
\begin{proof}
For $z = \exp(i\theta)$,	$|\frac{1-|a|^2}{z-a}|  \leq |\frac{1-|a|^2}{1-|a|}|  \leq 1+|a| \leq 2$.
\end{proof}

\begin{lemma}  \label{lemma:atom-lipschitz}
For any $a \in \D_\rho$ and $z_k=\exp(i\theta_k)$ we have 
\begin{align}
\label{eq:vary-freqs} |\varphi_a(z_1) - \varphi_a(z_2)| &\leq \frac{1+\rho}{1-\rho} |\theta_1-\theta_2|
\end{align}
\end{lemma}

\begin{proof}
\begin{align*}
|\varphi_a(z_1) - \varphi_a(z_2)|
& =   \left|\frac{1-|a|^2}{z_1-a}-\frac{1-|a|^2}{z_2-a}\right|\\
&\leq (1-|a|^2) \left|\frac{z_1-z_2}{(z_1-a)(z_2-a)} \right|\\
&\leq \frac{(1-|a|^2)}{(1-|a|)^2} |z_1-z_2| \leq \frac{1+\rho}{1-\rho} |\theta_1-\theta_2|\,.
\end{align*}
\end{proof}

\begin{lemma} For any $a,b \in \D_\rho$,
\begin{equation}\label{eq:hankel-op-bound}
	\|\Gamma_{\varphi_a}-\Gamma_{\varphi_b}\|_1 \leq \frac{2\rho}{1-\rho} |a-b|\,.
\end{equation}
\end{lemma}
\begin{proof}
	The Hankel operator for $\varphi_a(z)$ is given by the semi-infinite, rank one matrix
	\[
		(1-|a|^2)\left[\begin{array}{ccccc}
			1 & a & a^2 & a^3 & \cdots\\
			a & a^2 & a^3 & a^4 & \cdots\\
			a^2 & a^3 & a^4 & a^5 & \cdots\\
			\vdots & \ddots
		\end{array}\right] = (1-|a|^2) \left[\begin{array}{c} 1 \\ a\\ a^2\\a^3 \\ \vdots \end{array}\right]
		\left[\begin{array}{c} 1 \\ a\\ a^2\\a^3 \\ \vdots \end{array}\right]^T\,.
	\]
Let $\zeta_a = \sqrt{1-|a|^2} \left[\begin{array}{cccccc} 1 & a & a^2 & a^3 & \cdots \end{array}\right]^T$.  Note that $\zeta_a \in \ell_2$ with norm equal to $1$.  Also note that we have
\begin{equation}\label{eq:zeta-dot}
	\langle \zeta_a, \zeta_b \rangle =  \frac{\sqrt{1-|a|^2}\sqrt{1-|b|^2}}{1-\bar{a}b}\,.
\end{equation}
Then we have
\begin{subequations}
\begin{align}
\nonumber	\|\Gamma_{\varphi_a}-\Gamma_{\varphi_b}\|_1 &= \| \zeta_a \zeta_a^T - \zeta_b \zeta_b^T\|_1\\
	& =  \| \zeta_a (\zeta_a-\zeta_b)^T +  (\zeta_a-\zeta_b) \zeta_b^T\|_1\\
	\label{eq:tri-ineq} &\leq \| \zeta_a (\zeta_a-\zeta_b)^T \|_1 + \| (\zeta_a-\zeta_b) \zeta_b^T\|_1\\
	\label{eq:nuc-2-l2} &= 2 \| \zeta_a-\zeta_b\|_{\ell_2}\\
\label{eq:l2-dist-formula}	&=2 \sqrt{2}\sqrt{1 - \mathfrak{Re} \frac{\sqrt{1-|a|^2}\sqrt{1-|b|^2}}{1-\bar{a}b}}\\
	\nonumber&\leq \frac{2 \rho}{1-\rho} |a-b|
\end{align}
\end{subequations}
Here,~\eq{tri-ineq}  is the triangle inequality.~\eq{nuc-2-l2} follows because the nuclear norm of a rank one operator is equal to the product of the $\ell_2$ norm of the factors.~\eq{l2-dist-formula} follows from~\eq{zeta-dot}.  The final inequality follows from analyzing the taylor series of the preceding expression.
\end{proof}

\section{Dual norms}
We record here a few basic properties about dual atomic norms that we need for our proofs (see~\cite{CRPW10,BhaskarAllerton11} for more details).  For an atomic set $\cA$, the dual norm is given by
\[
	\| z \|_{\cA}^* = \sup_{a \in \cA} \langle a, z \rangle\,.
\] 
Note that for this norm, we have the generalization of H\"{o}lder's inequality $\langle x,z\rangle \leq \|x\|_{\cA}\|z\|_{\cA}^*$.  Moreover, note that we have the chain of inequalities
\[
	 \alpha \|x\|_{\cA'} \leq \|x\|_{\cA} \leq \beta\|x\|_{\cA'}
\]
for some $\alpha\leq 1$ and $\beta\geq 1$ for all $x$ if and only if
\[
	 \beta^{-1} \|z\|_{\cA'}^* \leq \|z\|_{\cA}^* \leq \alpha^{-1}\|z\|_{\cA'}^*
\]
for all $z$.  

\section{Proof of Proposition~\ref{prop:grid}} First note that for any atomic sets $\cA \subset \cA'$, $\|x\|_{\cA'} \leq \|x\|_{\cA}$.  The harder part of this proposition is the lower bound.  To proceed, we use the dual norm. 

Let $\D_\rho^{(\tau)}$ be a subset of $\D_\rho$ such that for every $a \in \D_\rho$, there exists an $\hat{a}\in \D_\rho^{(\tau)}$ satisfying $|a-\hat{a}|\leq \tau$.  For each $a \in \D_\rho$, we will actually denote $\hat{a}$ as the closest point in $\D_\rho^{(\tau)}$ to $a$.

Now observe that 
\begin{align*}
	  \| \cL(\varphi_{\hat{a}}-\varphi_a)\|_{\cL(\cA)} \leq  \| \varphi_{\hat{a}}-\varphi_a\|_{\cA}
	  \leq \frac{8}{\pi} \|\Gamma_{\varphi_{\hat{a}}}- \Gamma_{\varphi_a}\|_1 \leq \frac{16 \rho \tau}{\pi (1-\rho)}\,.
\end{align*}
Here, the first inequality follows from our reasoning in Section~\ref{sec:computation}.  The second inequality is Theorem~\ref{thm:hankel-nuclearity}, and the final inequality is by~\eq{hankel-op-bound}.

We then can compute
\begin{align*}
\|z\|_{\cL(\cA)}^* &= \sup_{a \in \D_\rho} \langle \cL(\varphi_a),z \rangle\\
&= \sup_{a \in \D_\rho} \langle \cL(\varphi_{\hat{a}}), z \rangle + \langle \cL(\varphi_{a}-\varphi_{\hat{a}}), z\rangle \\
&\leq\sup_{a \in \D_\rho^{(\tau)}} \langle \cL(\varphi_{\hat{a}}), z \rangle +\sup_{a \in \D_\rho}\langle \cL(\varphi_a-\varphi_{\hat{a}}), z\rangle \\
&= \|z\|_{\cL(\cA_\tau)}^*+ \sup_{a \in \D_\rho}\langle \cL(\varphi_a-\varphi_{\hat{a}}), z\rangle \\
&\leq \|z\|_{\cL(\cA_\tau)}^*+ \sup_{a \in \D_\rho}  \| \cL(\varphi_{\hat{a}}-\varphi_a)\|_{\cL(\cA)} \|z\|_{\cL(\cA)}^*\\
&\leq \|z\|_{\cL(\cA_\tau)}^*+ \frac{16\rho\tau}{\pi(1-\rho)} \|z\|_{\cL(\cA)}^*\,.
\end{align*}
Rearranging both sides of this inequality gives
\[
	\|z\|_{\cL(\cA)}^* \leq C_\tau^{-1} \|z\|_{\cL(\cA_\tau)}^*
\]
with $C_\tau =  1-\tfrac{16\rho\tau}{\pi(1-\rho)}$, completing the proof.

\section{Optimality conditions for DAST}
The following two important inequalities were proven in~\cite[Section 2]{BhaskarAllerton11}:
\begin{theorem}\label{thm:badri}
Let $\mathcal{Q}\subset \R^n$ be an arbitrary set of atoms.  Suppose that we observe $y = x_\star+\omega$ where $\omega\sim \mathcal{N}(0,\sigma^2 I)$.  Let $\hat{x}$ denote the optimal solution of
\[
	\minimize_x \tfrac{1}{2}\|x-y\|_2^2 + \mu \|x\|_\mathcal{Q}
\]
with  $\mu \geq\|\omega\|_{\mathcal{Q}}^*$.  Then we have
\begin{align}
	 \|\hat{x}-x_\star\|_2^2 &\leq 2 \mu \|x_\star\|_{\mathcal{Q}}\\
	 \|\hat{x}\|_{\mathcal{Q}} &\leq \|x_\star\|_{\mathcal{Q}} + \mu^{-1} \langle \omega, \hat{x}-x_\star\rangle\,.
\end{align}
\end{theorem}
We will use these inequalities in the following proof

\section{Proof of Theorem~\ref{thm:estimation}} 
To upper bound the $\cH_2$ norm, let us use some properties of functions which admit atomic decompositions.  Note that for any function, $H(z) = \sum_{w \in \D} c_w \varphi_w(z)$ with $\|H\|_{\cA}=\sum_{w\in\D} |c_w|$, we have, for any $z_1,z_2\in\bbS$,
\begin{subequations}
\begin{align}
\nonumber	|H(z_1)|^2-|H(z_2)|^2 &= 
	(|H(z_1)|-|H(z_2)|)(|H(z_1)|+|H(z_2)|) \\
\nonumber	&\leq 2 |H(z_1)-H(z_2)| \|H\|_{\cH_{\infty}} \\
\label{eq:tri-ineq2}	&\leq 2 \left( \sum_{w \in \D} |c_w| |\varphi_w(z_1)-\varphi_w(z_2)|\right) \left(\sum_{w\in \D} |c_w| \|\varphi_w\|_{\cH{\infty}}\right)\\
\label{eq:applylemmas}	&\leq 4 \frac{1+\rho}{1-\rho}  \left( \sum_{w \in \D} |c_w|\right)^2 |\theta_1-\theta_2|\\
\nonumber	 &= 4 \frac{1+\rho}{1-\rho}  \|H\|_{\cA}^2 |\theta_1-\theta_2|\,.
\end{align}
\end{subequations}
Here,~\eq{tri-ineq2} is the triangle inequality and~\eq{applylemmas} uses Lemmas~\ref{lemma:atom-hinf} and~\ref{lemma:atom-lipschitz}.

Let $\Delta = G_\star-\hat{G}$ and $\theta_k= \frac{2 \pi k}{n}$.  Then we can bound the norm of $\Delta$ as
\begin{align*}
\|\Delta\|_{\cH_2}^2 &= \frac{1}{2\pi} \int_0^{2\pi} |\Delta(e^{i\theta})|^2 d\theta\\
&= \frac{1}{2\pi} \sum_{k=0}^{n-1} \int_{\theta_k}^{\theta_{k+1}} |\Delta(e^{i\theta})|^2 d\theta\\
&\leq \frac{1}{2\pi} \sum_{k=0}^{n-1} \int_{\theta_k}^{\theta_{k+1}} \left(|\Delta(e^{i\theta_k})|^2 + 4 \frac{1+\rho}{1-\rho} \|\Delta\|_{\cA}^2 |\theta-\theta_k| \right)d\theta\\
&= \frac{1}{n}\sum_{k=0}^{n-1} |\Delta(e^{i\theta_k})|^2 + \frac{4\pi}{n} \frac{1+\rho}{1-\rho} \|\Delta\|_{\cA}^2\,.
\end{align*}
The inequality here follows from our preceding argument. 

Now, in this expression, we need to both upper bound the size of $\Delta$ on the measured frequencies and its atomic norm.  We will bound the latter in terms of the former:
\begin{subequations}
\begin{align}
\label{eq:tri-ineq3}	\|\Delta\|_{\cA} &\leq \| G_\star\|_{\cA}+\|\hat{G}\|_{\cA}\\
\label{eq:Ghat-def}	& = \|G_\star\|_{\cA} + \|\cL(\hat{G})\|_{\cL(\cA_\epsilon)}\\
\label{eq:apply-dast-cond}	& \leq \|G_\star\|_{\cA} + \|\cL(G_\star)\|_{\cL(\cA_\epsilon)} + \mu^{-1} \langle \omega, \cL(\hat{G}-G_{\star})\rangle \\
\label{eq:apply-grid-prop}	& \leq \|G_\star\|_{\cA} + (1-\delta)^{-1}\| \cL(G_\star)\|_{\cL(\cA)} + \mu^{-1} \langle \omega, \cL(\hat{G}-G_{\star})\rangle \\
\label{eq:inf-domination}		& \leq \frac{2-\delta}{1-\delta} \|G_\star\|_{\cA}  + \mu^{-1} \langle \omega, \cL(\hat{G}-G_{\star})\rangle \\
\label{eq:holder-mf}		& \leq \frac{2-\delta}{1-\delta} \|G_\star\|_{\cA}  + \mu^{-1} \| \omega\|_2 \left(\sum_{k=1}^{n-1} |\Delta(e^{i\theta_k})|^2\right)^{1/2} \,.
\end{align}
\end{subequations}
\eq{tri-ineq3} is the triangle inequality. \eq{Ghat-def} follows from how we defined $\hat{G}$.  \eq{apply-dast-cond} follows from Theorem~\ref{thm:badri}.  \eq{apply-grid-prop} follows from Proposition~\ref{prop:grid}.  \eq{inf-domination} follows because $\|\cL(H)\|_{\cL(\cA)} \leq \|H\|_{\cA}$ for all linear maps $\cL$ and transfer functions $H$.  \eq{holder-mf} is H\"{older}'s inequality.  Note that the quantity $\|\cL(G_\star)\|_{\cL(\cA_\epsilon)}$ could be infinite if the set $\{\cL(\phi_a)~:~a\in \D_{\rho}^{(\epsilon)}\}$ does not span $\R^n$.  This is precisely what leads to us including this assumption in the theorem statement.

To bound the size of $\Delta$ on the frequency grid, we use Theorem~\ref{thm:badri}:
\begin{align*}
	\frac{1}{n}\sum_{k=1}^{n-1} |\Delta(e^{i\theta_k})|^2  &\leq \frac{2\mu}{n} \|\cL(G_\star)\|_{\cL(\cA_\epsilon)}
	\leq \frac{2\mu}{n} (1-\delta)^{-1}\|G_\star\|_{\cA}	\,.
\end{align*}

Let $\omega \sim \mathcal{N}(0,\sigma^2 I_n)$. Using the well known upper bound for maximum of Gaussian variables (see, for example \cite{lr76}), we have
\begin{align*}
\E[\|\omega\|_{\cL(\cA_\epsilon)}^*] &=	\E\left[ \sup_{a\in\D_\rho^{(\epsilon)}} \langle \cL(\varphi_a), \omega \rangle\right] \leq \sigma \left(\sup_{a\in \D_\rho}\|\cL(\varphi_a)\|_2\right) \sqrt{2 \log |\D_\rho^{(\epsilon)}|}\,.
\end{align*}
Now, $\|\cL(\varphi_a)\|_2 \leq \sqrt{n} \|\varphi_a\|_{\cH_\infty}\leq 2 \sqrt{n}$.  Moreover, by a simple volume argument, $|\D_\rho^{(\epsilon)}| \leq \frac{1024 \rho^4}{\pi^2(1-\rho)^2\delta^2}$.  To see this, suppose $\mathcal{S}$ is a maximal set of points on $\D_\rho$ which are separated by at least $\tau$.  The maximal size of such a set is at most  $\frac{4\rho^2}{\tau^2}$. Moreover, $|\mathcal{S}|\geq |\D_\rho^{(\delta)}|$.  Now set $\tau = \epsilon=\tfrac{\pi (1-\rho) \delta}{16\rho}$.  In particular, note that we have $\E[\|\omega\|^*_{\cL(\cA_\epsilon)}] =\tfrac{1}{2} \mu$.

Now we can put all of the ingredients together.
\begin{align*}
\|\Delta\|_{\cH_2}^2  &\leq \left(1+\frac{8\pi \|\omega\|_2^2}{\mu^2}\frac{1+\rho}{1-\rho} \right)\frac{1}{n}\sum_{k=1}^{n-1} |\Delta(e^{i\theta_k})|^2\\
&\qquad + \frac{16\pi}{n(1-\delta)^2} \frac{1+\rho}{1-\rho}  \|G_\star\|_{\cA}^2\\
&  \leq \left(1+2\pi \frac{1+\rho}{1-\rho} \right) \frac{8\sigma}{1-\delta} \sqrt{\frac{2\log(\frac{1024 \rho^4}{\pi^2(1-\rho)^2\delta^2})}{n}}\|G_\star\|_{\cA} \\
&\qquad+ \frac{16\pi}{n(1-\delta)^2} \frac{1+\rho}{1-\rho}  \|G_\star\|_{\cA}^2\\
&  \leq  59 \frac{1+\rho}{1-\rho} \left( \sqrt{4\sigma^2 \log\left(\frac{11 \rho^2}{(1-\rho)\delta}\right)}\sqrt{\frac{\|G_\star\|_{\cA}^2}{n(1-\delta)^2} } + \frac{\|G_\star\|_{\cA}^2}{n(1-\delta)^2}   \right)
\end{align*}
as desired.

Applying the inequality $\|G_\star\|_{\cA}\leq \tfrac{8}{\pi}\|\Gamma_{G_\star}\|_1$ completes the proof.

\end{document}